\documentclass[preprint,12pt]{elsarticle}

 \usepackage{graphics}
 \usepackage{graphicx}
 \usepackage{epsfig}

\usepackage{amssymb}
\usepackage{bbm}

\usepackage[nodots,nocompress]{numcompress}

\usepackage{amsfonts}

\usepackage{amsthm}
\usepackage{amsmath}
\usepackage{amsfonts}
\usepackage{amscd,verbatim}
\usepackage{euscript}
\usepackage{enumerate}
\usepackage{amssymb}
\usepackage[latin1]{inputenc}
\numberwithin{equation}{section}
\newtheorem{thm}{Theorem}[section]
 \newtheorem{cor}[thm]{Corollary}
 \newtheorem{lemma}[thm]{Lemma}
 \newtheorem{prop}[thm]{Proposition}
 \theoremstyle{definition}
 \newtheorem{definition}[thm]{Definition}
 \theoremstyle{remark}

\begin{document}

\begin{frontmatter}


\title{ Decomposition of an $ L^{1}(T) $-bounded martingale and Applications in Riesz spaces}
%
\author[rvt]{Mounsif NIOUAR\corref{cor1}}
\ead{mounssif.niouar@uit.ac.ma}
\author[rvt]{Tarik BOUKARA    \corref{}}
\ead{tarik.boukara@uit.ac.ma}
\author[rvt]{Kawtar RAMDANE   \corref{}}
\ead{kawtar.ramdane@uit.ac.ma}
\author[rvt]{Youssef BENTALEB \corref{}}
\ead{youssef.bentaleb@uit.ac.ma}

\cortext[cor1]{Corresponding author}
\address[rvt]{Science of  Engineering Laboratory,  National School of Applied Sciences of Kenitra , Ibn Tofail University, Kenitra, Morocco.}

\begin{abstract}
In this paper, we give a decomposition of a martingale into three martingales with applications to certain types of inequalities in the new theory of Stochastic Analysis in Vector Lattices.\quad\quad\quad\quad\quad\quad\quad\quad\quad\quad\quad\quad
\end{abstract}


%

\end{frontmatter}


\section{Introduction and main results}

When the Theory of Stochastic Analysis on Riesz Spaces was first initially introduced, much attention was given to this theory, particularly by the South African and Tenisian schools see, e.g., \cite{ARC,TR,WLC,WLD}. In this paper, a significant contribution is made to the program, focusing on a major decomposition for bounded  martingales in martingale Theory. 

The decomposition theorem states in probability theory that for an $L^1$-bounded martingale $f$, it can be decomposed into three martingales a,b and d, such that $f=u+v+w$. The martingale $u$ is $L^1$-bounded and has a small increment sequence, while $v$ is absolutely convergent, and $w$ is uniformly bounded. This decomposition can be used to obtain inequality for a certain class of random variables. The mappings $f^*$ and $ S(f)$ belong to this class, where $f$ is a martingale. If $L$ is a  mapping belonging to class $\mathcal{A}$ and $f$ is an $L^1$-bounded martingale, then $ \mathbb{P}(\vert Lf\vert > \lambda ) \leq C\Vert f\Vert / \lambda $. This inequality has been proven thought the previous decomposition of the martingale. Besides this decomposition provides also a direct proof of certain inequalities due to Burkholder, eliminating the need for his indirect and difficult technique for establishing maximal inequalities. The presentation of the decomposition is self-contained and does not require any additional knowledge beyond the standard lore of martingale theory.

In Section 3, we complete the Riemann integral in Riesz spaces introduced in \cite{BI} with a certain property and H\"older's theorem, which is the aim of this section. In addition, the latter theorem enables us to prove the proposition in the next section.

In section 4, the context of martingale theory in Riesz spaces, the paper introduces a class of quasi-linear operators from $E_u^{\mathbb{N}}$ to $E_s$, referred to as class $\mathcal{A}$. The main result of the paper is the decomposition theorem of a martingale in Riesz space, where for an $L^{1}(T)$-bounded martingale $f$ and a constant $C$, the martingale $f$ can be decomposed into thee martingales $a,b$ and $d$, satisfying specific conditions. The paper also presents Theorem 4.4, which provides  a direct proof of certain inequalities, and concludes by demonstrating an inequality related to the notion of transform martingale in Riesz space. The paper relies on books as the sole sources of unexplained terminology and notation on Probability Theory and Riesz Spaces.

\section{Preliminaries}
In this paper, $E$ is considered   a Dedekind complete Riesz space with weak order unit $e>0$. A linear order-continuous projection $T: E \longrightarrow E$, strictly positive, such that range $R(T)$ of $T$ forms a Dedekind complete Riesz subspace of $E$ and $Te=e$   is called a conditional expectation. 
 A filtration in $E$ is a family of conditional expectations  $\left(T_{i}\right)_{i \geq 1}$ with $T=T_1$ and $T_{i} T_{j}=$
$T_{j} T_{i}=T_{j}$ for every $j \leq i$.  A stopping time adapted to the filtration  $\left(T_{i}\right)_{i \geq 1}$ is defined by an increasing sequence $\left(P_{i}\right)_{i \geq 1}$ of band projections on $E$, satisfying the condition $P_{i} T_{j}=T_{j} P_{i}$ for all $1 \leq i \leq j$, 
see \cite{WLD}. If $\left(f_{i}\right)$ is an increasing sequence with $f_{i} \in R\left(T_{i}\right)^{+}$ for $i=1,2, \ldots$, then the sequence $\left(P_{i}=P_{f_{i}}\right)_{i \geq 1}$ is a stopping time. Now, let
$\left(P_{i}\right)_{i \geq 1}$ be a stopping time adapted to the filtration $\left(T_{i}\right)_{i \geq 1}$ on $E$. Set 
$
\Delta P_{1}=P_{1}$ and $\Delta P_{i}=P_{i}-P_{i-1}$ for each  $i \in\{2,3, \ldots\}$. In \cite{WLD}, Definition 4.2] 
if the stopping time $P=\left(P_i\right)_{i \geq 1}$ is bounded, then the stopped process $\left(f_P, T_P\right)$ is defined by 
$$
f_P=\sum_{i=1}^{\infty} \Delta P_i f_i \quad \text { and } \quad T_P=\sum_{i=1}^{\infty} \Delta P_i T_i .
$$
 Besides in \cite{BT}
$$
f_{P-1\wedge n }=\sum_{i=1}^{n-1} \Delta P_i f_{i-1}+P_{n-1}^d f_{n-1} .
$$
In addition if $(T_{i})_{i\geq 1}$ is a filtration and $f_{i} \in R\left(T_{n}\right)$ with
$$
T_{i}\left(f_{j}\right)=f_{i} \quad \text { for all } i, m \text { with } i \leq j \text {, }
$$
then the family $\left(f_{i}, T_{i}\right)_{i\geq 1}$ defines a martingale on $E$.
 
In this paper, we define $E_{u}$ as the universal completion of $ E $,  if  $E_{u}$ satisfies two conditions:
 $E_{u}$ is universally complete, meaning it is  Dedekind complete and every subset of $E$ containg mutually disjoint elements possesses a supremum,
 and $E$ is contined in $E_{u}$  as an order-dense Riesz subspace. If $E$ is   archimedean Riesz space, then there exists    a unique universal completion of $E$ (up to a Riesz isomorphism) such that  $ e $ is a weak order unit for $ E_u $ if  $ e $ is a weak order unit for $ E$. Moreover, we extend the multiplication on the order ideal $ E_{e} $ (generated by $e$ in $E$) to $ E_u $  providing $ E_u $  an $ f $-algebra structure in which $ e $ is both multiplicative unit and weak order unit. This multiplication is also order continuous, see  \cite{ZRS}.
  
  Let $E$ be Dedekind complete Riesz space. The sup-completion $E_s$ of $E$ is a unique Dedekind complete ordered cone that has a biggest element, and  $E_s$ contains $E$ as a sub-cone equal to its group of invertible elements, see \cite{DES,GRJ}. Additionally, if $f \in E_s$, then $f=\sup \{g \in E: g \leq f\}$. Now, let $ f $  be a positive element in $E_s$, and $T$ be a strictly positive conditional expectation on $E$. The definition of $Tf$ is given  by $Tf= \sup T f_\alpha \in E_s$ for every  increasing net $(f_\alpha)\in E^+$  such that $f_\alpha \uparrow f$. 
   
Recall that, from \cite{WLC}, the natural domain of $ T $  is a Dedekind complete Riesz space with a weak order unit $e>0$ that satisfies $T e=e$ and is denoted by $L^{1}(T)$. In this paper, $L^{1}(T)$ denotes  the natural domain of $ T $. Notice  that $L^{1}(T)$ is a Dedekind complete Riesz space with a weak order unit $e>0$, and $T e=e$, for more detail, see \cite{WLC}. Grobler in \cite{GRJ} defines the $p$-power $f^{p}$ for $p \in(1, \infty)$ and $f$ an elements in $L^{p}(T)^{+}$. It should be pointed out that $f^{p}$ lies in the universal completion of $L^{1} (T)$.  The work in  \cite{TR} establishes  that 
$$
L^p(T)=\left\{f \in L^1(T):\vert f\vert ^p \in L^1(T)\right\}
$$
and
$$
\Vert f\Vert_{p}=T\left(\vert f\vert ^p\right)^{1 / p} \text { for all } f \in L^p(T)
$$
tfor each $p\neq \infty$.
Since, if $ p=\infty $ we have
$$
L^{\infty}(T)=\left\{f \in L^1(T):\vert f\vert  \leq u \text { for some } u \in R(T)\right\}
$$
and
$$
\Vert f\Vert_{\infty}=\inf \{u \in R(T):\vert f\vert  \leq u\} \text { for all } f \in L^{\infty}(T) .
$$
For a martingale  $f=(f_1, f_2,...)$   we define  $ \Vert f\Vert_p= \displaystyle\sup_n \Vert f_n \Vert_p$  $\in E_s$ for each $p\in [1,\infty]$. 
\section{ Riemann integral}
In this section, we  prove several properties of the Riemann integral as defined in \cite{BI}. One of these properties is H\"older's inequality, which  plays a vital role in establishing one of the key findings in this paper.
To provide context, let's recall that $E$ is a Dedekind complete Riesz space with a weak order unit, denoted as $e > 0$. When dealing with two real numbers $a$ and $b$, where $a < b$, we define a function $f : [a, b] \longrightarrow E$. We say that this function is bounded if there exists an element $M \in E^{+}$ such that $\vert f(x)\vert \leq M$ for all $x \in [a, b]$.

Let $f : [a, b] \longrightarrow E$ be a bounded function, and let $\alpha = \{a = \alpha_0 < \cdots < \alpha_n = b\}$ be a partition of the interval $[a, b]$. The mesh of $\alpha$ is defined as the maximum of the differences between consecutive partition points: 
\[
\Vert \alpha \Vert = \max\{\alpha_i - \alpha_{i-1} : 1 \leq i \leq n\}.
\]

For each $i$ in the range $1$ to $n$, we define $M_i(f)$ as the supremum  of $f(x)$ for $x$ in the subinterval $[\alpha_{i-1}, \alpha _i]$, and $m_i(f)$ as the infimum  of $f(x)$ in the same subinterval. We can then calculate the upper and lower sums of $f$ with respect to the partition $\alpha_n$ as follows:
\[
U(f, \alpha) = \sum_{k=1}^n M_i(\alpha_i - \alpha_{i-1})
\]
\[
L(f, \alpha) = \sum_{k=1}^n m_i(\alpha_i - \alpha_{i-1})
\]

Additionally, if we have two partitions $\alpha$ and $\beta$ where $\beta$ is finer (contains more subintervals) then $\alpha$, we observe the following relationships:
\[
L(f, \alpha) \leq L(f, \beta) \leq U(f, \beta) \leq U(f, \alpha).
\]

Given that $E$ is Dedekind complete, we can conclude that the limits:
\[
L(f) =  \sup L(f, \alpha)
\]
and
\[
U(f) = \inf U(f, \alpha)
\]
exist, with the supremum and infimum taken over all possible partitions of $[a, b]$. Importantly, it follows that $L(f) \leq U(f)$.
\begin{definition}[\cite{BI}, Definition 1].\label{defi 3.1} Let $a, b$ be two real numbers with $a < b$. A bounded function $f : [a, b] \longrightarrow E$  is considered Riemann integrable if and only if its lower Riemann sum $L(f)$ is equal to its upper Riemann sum $U(f)$.\\
The common value of these sums is denoted as $\int_{a}^{b}f(t)dt$ (or simply $\int_{a}^{b}f$).
\end{definition}
The set of all Riemann integrable functions from $ [a, b] $ to $ E $ is lattice-ordered and denoted by $ \mathcal{RI}([a,b],E) $ in this work. 
\begin{lemma}\label{lemma 3.2}
Let $f$ be a function in $\mathcal{RI}([a, b], E)$ and $k$ be a constant in $(0, \infty)$. Assume $g : [a, b] \longrightarrow E$ is a function that meets the following criteria:
$$\vert g(x) - g(y)\vert \leq k \vert f(x) - f(y)\vert  \quad \text{for all } x, y \in [a, b].$$
Then the function $g$ is Riemann integrable.
\end{lemma}
\begin{proof}
 Let $ f$ be a function Riemann integrable. Choose a fixed partion $\alpha=\{a=x_0<x_1< \cdots < \alpha_n=b\}$ of $[a, b]$,  
we have
$$
U\left(f, \alpha\right) - L\left(f, \alpha\right) = \sum_{i=1}^n \left(M_i(f) - m_i(f)\right)\left(x_i - x_{i-1}\right),
$$
and  for all $0 \leq i \leq n$ and $x_{i-1} \leq x, y \leq x_i$
   $$ \vert f(x)- f(y)\vert \leq M_i(f)-m_i(f) $$
 This implies that
$$
\left\vert g(x) - g(y) \right\vert \leq K\left(M_i(f) - m_i(f)\right).
$$
Then
$$
\left(M_i(g) - m_i(g)\right) \leq k \left(M_i(f) - m_i(f)\right).
$$
It follows that
$$
0 \leq U\left(g, \alpha\right) - L\left(g, \alpha\right) \leq k \left(U\left(f, \alpha\right) - L\left(f, \alpha\right)\right).
$$
Now, consider two partitions $ \sigma$ , $\delta$ of $[a,b]$, and putting $\alpha= \sigma \cup \delta$. Therefore 
  $$\begin{aligned}
0\leq U(g) - L(g)
&\leq U\left(g, \alpha\right) - L\left(g, \alpha\right)\\
 &\leq k \left(U\left(f, \alpha\right) - L\left(f, \alpha\right)\right)\\
  &\leq k \left(U\left(f, \sigma\right) - L\left(f, \delta\right)\right)
\end{aligned}
$$
Take the infimum over $\sigma$ and supremum over $\delta$, we obtain $ U(g)= L(g)$.
 \end{proof}
 We purposefully choose the Riesz space $E_e$ over $E$ in the following propositions because $E_e$ is an Archimedean $f$-algebra equipped with commutative multiplication. It is a well-known fact that multiplication is commutative in every Archimedean $f$-algebra. which help us to establish certain results such as the  H\"older theorem, particularly when we use the function defined as follows  $f: [a, b] \longrightarrow E_e$ . The Riesz space of Riemann integrable functions defined on the interval $[a, b]$ with valued in $ E_e$ is denoted by the set $\mathcal{RI}([a, b], E_e)$.
\begin{prop}\label{propoistion 3.3}
 Let   $ f,g $ be two functions  in $  \mathcal{RI}([a,b],E_{e}) $ and $u$ be an element  in $E_{e}$, then we have the following 
 \begin{enumerate}
 \item[(i)] $u f$ is  Riemann integrable and $\int_a^b u f(t)dt = u\int_a^b f(t)dt$. 
 \item[(ii)] $ f^{2} $ is Riemann integrable.
  \item[(iii)] $ fg $ is Riemann integrable.
\end{enumerate} 
\end{prop}
\begin{proof}

$(i)$ Obvious by (\cite{BI}, Theorem 4).
  
$(ii)$ Let $ f $ be a   Riemann integrable function on $ [a,b] $, then there exist a real number non-negative $ 
 \lambda  $  such that 
$$ \vert f(x)\vert \leq \lambda e  \quad \text{for each } x \in [a,b]. $$ 
On the other hand, for every $ x,y \in [a,b] $ we have that
$$ \vert f^{2}(x)-f^{2}(y)\vert \leq 2\lambda\vert f(x)-f(y)\vert  .$$
By  Lemma \ref{lemma 3.2}, this completes the proof.

$(iii)$  Observe that by the following remark,  the product $fg$ can be expressed  as follows: 
$$ fg=\frac{1}{2}((f+g)^2 - (f^2
+g^2))$$
\end{proof}

Now we're interested to prove that every Riemann integrable  function $f^p$ is also Riemann integrable for every real number $p> 0$. Note her that we use the forme $x^p$ introduced by Grobler in \cite{GRJ}. To this we need to prove the flowing technical lemma. 
\begin{lemma}\label{lemma 3.4}
Let $x$ and $y$ be tow  positives elements in $ E_{e}$. We have the following inequalities
\begin{enumerate}
 \item[(1)] $
\vert x^{p}-y^{p} \vert \leq 
		p\vert x-y\vert ( x^{p-1}+y^{p-1})   
$
for each real number $p\geq 1$.
\item[(2)] $ \vert x^{p}-y^{p} \vert \leq \vert x-y\vert ^p$ for each real number $ 0<p<1$.
\end{enumerate}
\end{lemma}
\begin{proof} $(1)$ Assume that $ \beta e \leq  x,y \leq \mu e$ for some $\beta , \mu \in (0,\infty)$,  then $x,y$ are bounded in $E$ so, by [\cite{AB},  Lemma 2], $ x^{p}$ and $y^{p}  $ exists   for each real number $p>0$ and the inverse $y^{-1}$ of $y$ exist in $E_e$; see [\cite{HP}, Theorem 3.4]. \\
Putting  the function
 $$ f(z)= (p-1)z^p - pz^{p-1}+ pz-(p-1)e \quad \text{for each } e\leq z \in E_{e}.$$ 
 Note also that the real  function   $f(t)= (p-1)t^p - pt^{p-1}+ pt-(p-1)$  is continuous and increasing on $ [1,\infty )$.
 Therefore, from [\cite{AB}, Lemma 2], $ f $  is increasing function  for every $e\leq z$ in $E_{e}$ with $f(e)=0$.
 
  \textbf{Firstly}, choose that $x\geq y$. For $z=xy^{-1} \geq e$,  we obtain
$$
(p-1)(xy^{-1})^p - p(xy^{-1})^{p-1}+ p(xy^{-1})-(p-1)e\geq 0.
$$
Multiplying by $y^{p}$, we have that
 $$
(p-1)x^p - px^{p-1}y+ pxy^{p-1}-(p-1)y^p\geq 0.
$$
 Thus 
   $$     \vert x^p-y^p \vert \leq  p\vert x-y\vert (x^{p-1}+y^{p-1})  .$$ 
   
 \textbf{Secondly}, if $x$ and $y$ are arbitrary positives elements  in $E_e$, we put 
 $$u=x\vee y \quad \text{ and} \quad v=x\wedge y.$$
  Hence  
 $$     \vert u^p-v^p \vert \leq  p\vert u-v\vert (u^{p-1}+v^{p-1}) .$$ 
 Moreover,
 $$ \vert u^p
- v^p\vert=(x\vee y)^p -(x\wedge y)^p =x^p \vee y^p - x^p\wedge y^p=\vert x^p -y^p\vert, $$ 
and  
      $$ u^{p-1}+v^{p-1}= x^{p-1}+y^{p-1}.$$
It follows from this that 
$$\vert x^{p}-y^{p} \vert \leq
		p\vert x-y\vert ( x^{p-1}+y^{p-1}).   $$
$(2)$ Similarly	of $(1)$ with using the function $f(z)=(z-e)^p-(z^p-e)$.
\end{proof}
\begin{prop}\label{proposition 3.5} Let $ f $ be  an element of $\mathcal{RI}([a,b],E_{e})$, then 
  $ f^p $ is Riemann integrable  for all real number $ p\geq 1$.
 \end{prop} 
\begin{proof}
  We will show that $ (f^+)^p$ and $(f^-)^p$ are Riemann integrable. It is easy to see that if $ f$ in $\mathcal{RI}([a,b],E_{e})$, then $ f^+$ and $f^-$ in $\mathcal{RI}([a,b],E_{e})$. Note now that the $ f^+$ , $f^-$ are positives elements  verified the inequalities in Lemma \ref{lemma 3.4}.
So, by Lemma \ref{lemma 3.2},  $ (f^+)^{p} $  and  $ (f^-)^{p} $  are Riemann integrable functions. As observed earlier, we claim that $ f^p =  (f^+)^{p} - (f^-)^{p} $ in the $f$-algebras. We conclude $ f^{p} $  in $\mathcal{RI}([a,b],E_{e})$
\end{proof}
 
\begin{thm} $[$H\"older Inequality$]$.\label{theo 3.6}  
Let $ f , g $ be two functions  in a Riesz space $\mathcal{RI}([a,b],E_{e})$. Then we have 
$$ \vert \int_{a}^{b}fg\vert \leq \left(\int_{a}^{b}\vert f\vert^{p}\right)^{\frac{1}{p}}\left(\int_{a}^{b}\vert g\vert ^{q}\right)^{\frac{1}{q}}     $$
for each $1\leq p,q<\infty $ with $ \frac{1}{p}+\frac{1}{q}=1$.
\end{thm}
\begin{proof}
 Without losing generality, we assume  that $0 <f, g \in \mathcal{RI}([a,b],E_{e})$. Then there are two real numbers $\alpha, \beta \in(0, \infty)$ such that
$
\alpha e \leq  f(x), g(x) \leq \beta e \text { for any  } x \in [a,b]
$
and simply we  write $\alpha e \leq  f, g \leq \beta e$.\\
Thus, for each $p \in [1,\infty)$ we have
$$f^p \geq \alpha^{p}e>0,$$
 and
$$
\int_a^b f^p \geq \int_a^b \alpha^p e=(b-a) \alpha^p e>0.
$$
 Hence $\left(\int_a^b f^p\right)^{-1}$ exist in $E_{e}$ because  $(b-a) \alpha^p e$ is invertible in the $f$-algebra $ E_{e}$, see [\cite{HP}, Theorem 3.4]. \\
  Similarly, $\left(\int_a^b g^q\right)^{-1}$ existe in $ E_{e} $.
and by Young Inequality, see [\cite{TR}, Propostion 3.6]. We claim
$$
\left(\int_a^b g^q\right)^{\frac{1}{p}}\left(\int_a^b f^p\right)^{\frac{1}{q}} fg \leq \frac{1}{p} f^p \int_a^b g^q+\frac{1}{q} g^q \int_a^b f^p.
$$
Integrating the two members  and $(i)$ of Proposition  \ref{propoistion 3.3}.
$$
\left(\int_a^b g^q\right)^{\frac{1}{p}}\left(\int_a^b f^p\right)^{\frac{1}{q}} \int_a^b f g \leq \left(\int_a^b f^p\right)\left(\int_a^b g^q\right).
$$
This show,
 $$\ \int_a^b f g \leq\left(\int_a^b f^p\right)^{\frac{1}{p}}\left(\int_a^b g^q\right)^{\frac{1}{q}}.$$
\end{proof}

\section{ Main results}

We abstract these features and list  them under the heading "Class $\mathcal{A}$". If only for attention, this definition seems reasonable on the essential characteristics of the following arguments.
\begin{definition} \label{definition 4.1} Let  $E$   be a Dedekind complete Riesz space with weak order unit  $e$. an operator $L$ \; from  $E_u^{\mathbb{N}}$  to  $E_s$  is said to be of  class $ \mathcal{A} $ if: 
 
\begin{enumerate}
 \item[(1)] $L$  is quasi-linear, i.e $\vert L(f+g)\vert  \leq C(\vert L(f)\vert +\vert L(g)\vert )$ ,
  \item[(2)]  $TP_{\vert L(f) \vert} e \leq C TP_{f^{*}}e$,
\item[(3)] The mapping  $L$  satisfies the following inequalities
\begin{enumerate}
\item[(i)] If $ f=(f_{1},f_{2},...)  $ where $ f_{n}=\displaystyle \sum_{k=1}^{n}\Delta f_{k} $,  $ \text { then }\; \Vert L f\Vert_1  \leq C \Vert \displaystyle\sum_{k=1}^{n}\vert\Delta f_{k}\vert\Vert_1.$ 
\item[(ii)] $\Vert Lf\Vert_2 \leq C \Vert f\Vert_2$.
\end{enumerate}
\end{enumerate}
\end{definition}
Decomposition theorem 
 plays an important role in the demonstration of certain class of operators. 
\begin{lemma}\label{lemma 4.2} Let $ (P_{a_{i}}) $ ,$ (P_{b_{i}}) $ be an increasing  bands projections.
 If $ P_{i}= P_{a_{i}}\vee P_{b_{i}}  $ then $ \Delta P_{i}\leq \Delta P_{a_{i}}+\Delta P_{b_{i}}.  $   
\end{lemma} 
\begin{proof} It is easily to see that
$$
\begin{aligned}  \Delta P_i&=P_{a_{i}} \vee P_{b_{i}}-P_{a_{i-1}} \vee P_{b_{i-1}}\\
&=\Delta P_{a_{i}} + \Delta P_{b_{i}} -  P_{a_{i}} P_{b_{i}}+P_{a_{i-1}} P_{b_{i-1}} \\
 &\leq \Delta P_{b_{i}}+\Delta P_{a_{i}}.
\quad
\end{aligned} $$
\end{proof}

\begin{lemma}\label{lemma 4.3}
Let $ (x_k)_{k\geq 1}$ be a stochastic process in $E^+$. Put $$ \left( P_k= P_{\vee_{i=1}^k (x_i-\lambda e)^+}\right)_{k\geq 1} $$ for every  $\lambda >0$. Then  $ \Delta P_k \Delta x_k \geq 0$.
\end{lemma}

\begin{proof} Note that the case of $k=1$ is obvious and for each $k\geq 2$ we can write
$$\begin{aligned} \Delta P_k \Delta x_k & =P_k P_{k-1}^d\left(x_k-x_{k-1}\right) \\
 & =P_k P_{k-1}^d\left(x_k-\lambda e\right)^+ - P_k P_{k-1}^d\left(x_k-\lambda e\right)^-\\ &-P_k P_{k-1}^d\left(x_{k-1}-\lambda e\right)^+
  + P_k P_{k-1}^d\left(x_{k-1}-\lambda e\right)^-. \\  
  \end{aligned}$$
  In the other hand, we have 
         $$  \left(x_{k-1}-\lambda e\right)^+ \in B_{(x_{k-1}-\lambda e)^+}  \subset B_{k-1}.   $$
 This implies $$ P_{k-1}^{d} \left(x_{k-1}-\lambda e\right)^+ =0.$$
   Now, we show $ P_k P_{k-1}^d\left(x_k-\lambda e\right)^-=0$. Since $$ P^{d}_{(x_{k}-\lambda e)^+}(x_k-\lambda e)^-
= (x_k-\lambda e)^-.$$
 Then
 $$   P_k P_{k-1}^d\left(x_k-\lambda e\right)^-=P_k P_{k-1}^d P^{d}_{(x_{k}-\lambda e)^+}(x_k-\lambda e)^-= P_kP_k^d (x_k-\lambda e)^-=0. $$
This complete the proof.
\end{proof}

\begin{lemma}
Let $P=\left( P_n\right)$ sequence of band projecion such that $ P_n\uparrow P_\infty$ and $ (a_n)$ be sequence in  $E^+$. We have the following 

$$ P_\infty \left( \sum_{k=1}^{\infty} P^{d}_{k-1}(a_{k})\right)  = \sum_{k=2}^{\infty} \Delta P_k\left( \sum_{i=1}^{k-1} a_{i+1}\right) $$
\end{lemma}
\begin{proof}
 Its easily to see that
$$
\begin{aligned}
 \sum_{k=2}^{\infty} \Delta P_k\left(\sum_{i=0}^{k-1} a_{i+1}\right) 
& =\Delta P_2\left(a_2\right)+\Delta P_3\left(a_2+a_3\right)+\Delta P_4\left(a_2+a_1+a_3\right)+\cdots  \\
& =\sum_{n=2}^{\infty} \Delta P_n\left(a_2\right)+\sum_{n=3}^{\infty} \Delta P_n\left(a_3\right)+\sum_{n=4}^{\infty} \Delta P_n\left(a_4\right)+\cdots \\
& = (P_\infty -P_1)(a_2) + (P_\infty-P_2)(a_3)+\cdots \\
& = P_\infty \left( \sum_{k=1}^{\infty} P^{d}_{k-1}(a_{k})\right)
\end{aligned}
$$
\end{proof}
\begin{thm} \label{theo 4.3} Let $f$ be an $L^{1}(T)$-bounded martingale. Corresponding to all $\lambda>0$, the martingale $f$ may be decomposed into three martingales $u, v, w$ so that $f= u+v+w$ with
\begin{enumerate}
\item[i)] The martingale $u=(u_{1},u_{2},...)$ is $L^{1}(T)$-bounded, $\Vert u\Vert _{1} \leq C\Vert f\Vert _{1}$ and  
 $\lambda T P_{\Delta u^{*} }e \leq
C\Vert f\Vert _{1}$.
\item[ii)] The martingale $v=(v_{1},v_{2},...)$ where $v_n=\sum_{k=1}^{n}\Delta v_k$ is absolutely convergent and $\left\Vert \displaystyle\sum_{k=1}^{\infty}\left\vert \Delta v_{k}\right\vert \right\Vert _{1} \leq C\displaystyle\Vert f\Vert _{1}$.
\item[iii)] The martingale $w=(w_{1},w_{2},...) $ is uniformly bounded, $\Vert w\Vert _{\infty} \leq C \lambda e,\Vert w\Vert _{1} \leq C\Vert f\Vert _{1}$ and 
   $\Vert w\Vert _{2}^{2} \leq C \lambda\Vert f\Vert _{1}$.
\end{enumerate}
 \end{thm}
 
 \begin{proof}
 Let  $f$ be an  $L^{1}(T)$-bounded martingale. without loss of generality, we assume that $f$ is  non-negative. Indeed, if $f$ is of arbitrary sign then $f$ can be written as the sum of two non-negative martingales by Krickeberg's decomposition theorem, $f=g-h$ with $\Vert g\Vert _{1} \leq\Vert f\Vert _{1}$ and $\Vert h\Vert _{1} \leq\Vert f\Vert _{1}$, see \cite{GQV}.
 
Now, we define three stopping times. The first is given by
$$R=\left(P_{\vee_{k=1}^{n}\left( f_{k}-\lambda e\right)^{+}}\right)_{n \geq 1}.$$
    The second stopping time is as
 $$ S=\left(P_{\vee_{k=1}^{n} \left(g_k -\lambda e\right)^{+}}\right)_{n \geq 1}$$
where $g_n=\displaystyle\sum_{k=0}^{n}T_k(\epsilon_{k+1})$  and  $\varepsilon_{n}=\left(R_{n}-R_{n-1}\right)\left(\Delta f_{n}\right)$, from this definition of $ \varepsilon_n$ is a positive by Lemma \ref{lemma 4.3}. The last stopping time $\tau=R \vee S$. So, from
  Lemma \ref{lemma 4.2}
$$
\begin{aligned}
\Delta \tau_{n} e 
&\leq \Delta R_{n} e+\Delta S_{n} e.\\
\end{aligned}
$$
 Then, 
 $$ \sum_{n \geq 1} \Delta \tau_{n} e \leq  \sum_{n \geq 1} \Delta R_{n} e+  \sum_{n \geq 1} \Delta S_{n} e $$ 
 We will try to upper bound each term on the right side of the inequality by $ C\Vert f\Vert_{1}/\lambda $.\\
Note that
 $$ 
\lambda \Delta R_{n} e \leq \Delta R_{n} f_n
 .$$
Then by section 4 in \cite{BI}
$$
\begin{aligned}
\lambda T \sum_{n \geq 1} \Delta R_{n} e &\leq T f^{R} \leq C \Vert  f  \Vert _{1}. \\
\end{aligned}
$$
In the other hand
$$
\begin{aligned} \lambda \Delta S_{n} e & \leq  \Delta S_{n}\left(\sum_{k=0}^{n} T_{k}\left(\varepsilon_{k+1}\right)\right)\leq \Delta S_{n}\left(\sum_{k=0}^{\infty} T_{k}\left(\varepsilon_{k+1}\right)\right)
  \end{aligned}
$$
Thus,
$$
\begin{aligned} \lambda T \sum_{n \geq 1} \Delta S_{n} e & \leq  T \sum_{n \geq 1} \Delta S_n \left(\sum _ { k = 0 } ^ { \infty }  T _ { k } \left(\Delta R_{k+1}\left( \Delta f_{k+1} \right)\right)\right)\\ 
& \leq T \sup_n S_n \left(\sum _ { k = 0 } ^ { \infty }  T _ { k } \left(\Delta R_{k+1}\left( \vert \Delta f_{k+1} \vert \right)\right)\right)\\
&\leq T \left(\sum _ { k = 0 } ^ { \infty }  T _ { k } \left(\Delta R_{k+1}\left(  f_{k+1} + f_k \right)\right)\right)\\
 & \leq T f^{R}+T f^{R-1} \\
&\leq C  \Vert  f  \Vert _{1}.
\end{aligned}
$$
It follows that
$$
T\sum_{n \geq 1} \Delta \tau_{i} e \leq \frac{C}{\lambda}\left\Vert f\right\Vert _{1}.
$$

To show that the martingale $u$ satisfies these properties in (i). 	We set $$u=f-f^{\tau}.$$ 
Observe that
$
u_{n}=f_{n}-f_{n \wedge \tau}
  $ and  $ \Delta u_n= \tau_{n-1}\Delta f_n$.\\
This implies that   
  $$
  \Vert u\Vert _{1}=\displaystyle\sup_n T \vert u_{n}\vert                    \leq  2\Vert f\Vert _{1}, $$
  and we have 
  $$ \tau_{n-1}\Delta f_n \in B_{\tau_{n-1}e}.  $$
 Hence 
  $$ \sup_n P_{\tau_{n-1}\Delta f_n}e\leq \sup_n P_{\tau_{n-1}e}e \leq \sup_n \tau_n e= \sum_{n\geq 1} \Delta \tau_n e.$$
So
  $$ \lambda T P_{\Delta u^{*} }e \leq \lambda T \sum \Delta \tau_i e\leq
C\Vert f\Vert _{1}. $$
This completes the properties of $(i)$.

To construct the martingales $v$ and $w$, let's look at the martingale
 $ f^{\tau} $ defined as 
  $$ f_{n\wedge\tau}=\sum_{k=1}^{n}\tau_{k-1}^{d}\Delta f_{k}.$$
With 
$$
\tau_{k-1}^{d}=R_{k-1}^{d} S_{k-1}^{d}
.$$
Observe that
$$
\tau_{k-1}^{d} \Delta f_{k}=S_{k-1}^{d}\left(R_{k-1}^{d} \Delta f_{k}+R_{k}^{d} \Delta f_{k}-R_{k}^{d} \Delta f_{k}\right).
$$
 Set
$$
y_{k}=R_{k}^{d} \Delta f_{k} \quad , \quad \varepsilon_{k}=\left(R_{k-1}^{d}-R_{k}^{d}\right) \Delta f_{k}
,$$

$$
v_{n}=\sum_{k=1}^{n} S_{k-1}^{d}\left(\varepsilon_{k}-T_{k-1} \varepsilon_{k}\right),
$$
and
$$
w_{n}=\sum_{k=1}^{n} S_{k-1}^{d}\left(y_{k}+T_{k-1} \varepsilon_{k}\right).
$$
Clearly $$ f_{\tau \wedge n}=\sum_{k=1}^{n} S_{k-1}^{d}\left(y_{k}+\varepsilon_{k}\right)
= v_n +w_n $$
Let's show in short proof that $v$ and $w$ are martingales. Its obvious $ v_n$ , $w_n$ in $R(T_n)$ for each $n$ and we have:

$$
\begin{aligned}
   T_nv_{n+1} &= T_n\left( \sum_{k=1}^{n+1} S_{k-1}^{d}\left(\varepsilon_{k}-T_{k-1} \varepsilon_{k}\right)\right)\\
   &=  \sum_{k=1}^{n} S_{k-1}^{d}\left(T_n\varepsilon_{k}-T_nT_{k-1} \varepsilon_{k}\right) + T_n S_{n}^{d}\left(\varepsilon_{n+1}-T_{n} \varepsilon_{n+1}\right)
   \end{aligned}              $$ 
   Or $ f_k , f_{k-1} \in R(T_k) \subset R(T_n)$ implies $T_n \varepsilon_k= \varepsilon_k$. Then $T_nv_{n+1}=v_n$.\\
   For the proof of $w$, just  see that
 $$T_{k-1} y_{k}=T_{k-1}\left(y_{k}-R_{k-1}^{d} \Delta f_{k}\right)=-T_{k-1}\left(\varepsilon_{k}\right).$$
 And applying the even technical used in $v$.

Now, we aim to prove to absolute convergence of $v$ and its associated inequality. Indeed since  
 $$\sum_{k \geq 1} \Delta v_{k} \quad\text{where}\quad \Delta v_{k}=S_{k-1}^{d}\left(\varepsilon_{k}-T_{k-1}\left(\varepsilon_{k}\right)\right).$$
Then   
$$
\begin{aligned}
T\left(\sum_{k= 1}^{n}\left\vert \Delta v_k\right\vert \right)& \leq T\left(\sum_{k\geq 1} \mid S_{k-1 }^d\left(\varepsilon_k-T_{k-1} \varepsilon_{k} \right)\mid\right)\\
&\leq 2 T\left(\sum_{k=1} \Delta R_k f_k\right)\\
&\leq 2 T f^R\\
&\leq 2 \Vert  f \Vert. 
\end{aligned}
$$
 Since $L^1(T)$ is $T$-universally complete implies  the martingale $v$ converges absolutely and applying the supemum we get directly associated inequality in property $(ii)$.
 
Finally, we demonstrate  that the martingale $ w $ satisfies the third assumption .\\
 Indeed, for all $n \geq 1$
$$
\begin{aligned}
\sum_{k=1}^n y_k &=\left(I-R_{n}\right) f_n +\sum_{k=1}^{n-1}  \Delta R_{k+1} f_k\\
&= R_{n}^{d} f_n +\sum_{k=1}^{n-1}  R_{k+1}R_{k}^{d} f_k.
\end{aligned}
$$
So, 
$$
\begin{aligned}
\mid \sum_{k=0}^n y_k \mid &\leq R_{n}^{d} (\lambda e)  +\sum_{k=1}^{n-1}  R_{k+1}R_{k}^{d} ( \lambda e) \\
&= R_{n}^{d} (\lambda e)  +\sum_{k=1}^{n-1} ( R_{k+1}-R_{k}) ( \lambda e)\\ 
&= \lambda e,
\end{aligned}
$$
in addition
$$
\begin{aligned}
 \sum_{k=1}^{n} S_{k-1}^{d} T_{k-1}\left(\varepsilon_{k}\right)&= S_\infty \left(\sum_{k=1}^{n} S_{k-1}^{d} T_{k-1}\left(\varepsilon_{k}\right)\right) +  S_\infty^d \left(\sum_{k=1}^{n} S_{k-1}^{d} T_{k-1}\left(\varepsilon_{k}\right)\right) \\
 & \leq S_\infty\sum \Delta S_{k} \left(\sum_{i=1}^{k-1}  T_{i-1}\left(\varepsilon_{i}\right)\right) + \sum S_\infty^d T_{k-1}\varepsilon_k  \\
 &    \leq S_\infty\sum  S_{k}S_{k-1}^d \left(\sum_{i=1}^{k-1}  T_{i-1}\left(\varepsilon_{i}\right) +\lambda e-\lambda e\right) + S_\infty^d \left( \sum  T_{k-1}\varepsilon_k -\lambda e\right)   + S_\infty^d(\lambda e) \\
 &\leq  S_\infty (\lambda e) + S_\infty^d (\lambda e) +\lambda e \\
 &\leq 2\lambda e. 
\end{aligned}
$$
Then
$$
\begin{aligned}
 \vert w_n\vert   &\leq \vert \displaystyle\sum_{k=1}^n y_k\vert +\mid \displaystyle\sum_{k=1}^n S_{k-1}^d T_{k-1}\left(\varepsilon_k\right)\vert \\
 &\leq \lambda e+\lambda e \\
 &\leq 2 \lambda e.
 \end{aligned}
  $$
  Which implies $
\Vert w\Vert _\infty  \leq 2 \lambda e .
$\\
On the other hand
$$  
T\left(\displaystyle\sum_{k=1}^{n} y_k\right)=T\left(f_{R \wedge n}\right)
\leq \Vert f\Vert _1 .
$$
This gives 
 $$
   \sup_n T\left(\displaystyle\sum_{k=1}^{n} y_k\right) \leq C \Vert f\Vert_{1}$$
In addition 
$$  
 T\left(\sum^n_{k=1} T_{k-1} \varepsilon_k\right) =\sum^n_{k=1} T \varepsilon_k  
\leq T\left(\sum_{n \geq 1} \varepsilon_k\right) 
 \leq C\Vert f\Vert _1.
 $$
 Implies  
                 $$  \Vert w \Vert_1 = \sup_n w_n\ \leq  C\Vert f\Vert                $$
 
   It follows that $\Vert w\Vert_1 \leq C\Vert f\Vert_1$. 
For the last property, see that

$$ Tw_n^2= T(w.w)\leq \Vert w\Vert_\infty T(\vert w\vert) \leq 2\lambda T w_n $$ then $ \Vert w\Vert_2 \leq C\lambda \Vert f\Vert_1$
 
 \end{proof}

The following theorem follows from the previous theorem and the properties of mapping  Riesz space class $\mathcal{A}$.

\begin{thm}\label{theo 4.4}
 Let $f$ be an $L^{1}(T)$-bounded martingale and $L$ a mapping in the family of Riesz spaces class $\mathcal{A}$.Then we have
 $$\lambda T P_{(\vert L(f)\vert -\lambda \boldsymbol{e})^{+}} e \leq C\Vert f\Vert _{1}$$
for any $ \lambda >0. $
\end{thm}
\begin{proof}
  Let $f$ be an $L^{1}(T)$-bounded martingale. By Theorem \ref{theo 4.3}  we can write $$f=u+v+w.$$
 From property $(1)$ of class $\mathcal{A}$ we have  
 $$
 \begin{aligned}
 \vert L f\vert  &\leq C(\vert L u\vert +\vert L v\vert +\vert L w\vert ).
 \end{aligned}
$$
Moreover, for every $\lambda >0$ we claim
$$
  T P_{(\vert L(f)\vert -\lambda e)^+} e
  \leq T P_{\left(\vert L(u)\vert -\frac{\lambda}{3C} e\right)^{+}} e+T P_{\left(\vert L(v)\vert -\frac{\lambda}{3 C} e\right)^{+}} e+T P_{\left(\vert L(w)\vert -\frac{\lambda}{3 C} e\right)^{+}} e.
$$
Now, we try to prove that each term of the second member of this inequality is
increased by $\frac{C\Vert f\Vert _{1}}{\lambda}$.

\textbf{Firstly},

It is easy to see $ B_{\Delta u^*}^d \subset B_{u^*}^d$, additionally by the property $ (2) $ of class $ \mathcal{A} $ and $(i)$ of  Theorem \ref{theo 4.3} we show
$$
\begin{aligned}
TP_{\left(\vert L(u)\vert -\frac{\lambda}{3C} e\right)^{+}} e &\leq T P_{\vert L(u)\vert }e\\ 
&\leq C T P_{u^{*}}e\\
 &\leq \frac{C\Vert f\Vert _{1}}{\lambda}.
\end{aligned}
$$

\textbf{Secondly}, by  [\cite{GJC}, Lemma 3.1], property $(3 i)$ of class
$\mathcal{A}$ and $(i i)$ of  Theorem \ref{theo 4.3} we prove
$$T P_{(\vert L(v)\vert -\frac{\lambda}{3C} e)^+} e\leq \frac{C\Vert f\Vert _{1}}{\lambda}.$$

\textbf{Finally},  We have 
         $$ T P_{(\vert L(w)\vert -\frac{\lambda}{3C} e)^+}e  \leq C \frac{T\vert L(w)\vert^2}{\lambda ^2} $$

 property
$3 ii$ of the class $\mathcal{A}$, and $(i i i)$ of  Theorem \ref{theo 4.3} we have
 $$T P_{(\vert L(w)\vert -\lambda e)^+}e \leq \frac{C}{\lambda^{2}}\Vert  w \Vert _{2}^{2} \leq \frac{C\Vert f\Vert _{1}}{\lambda}.$$
 This complete  proof.
\end{proof}
By this theorem we get the following result 
\begin{cor}\label{Corollary 4.5}
Let $ f $ be an $ L^{1}(T)$-bounded martingale. Then 
\begin{enumerate}
\item[(i)]  $ \lambda TP_{(f^{*}-\lambda e)^{+}}e\leq C \Vert f\Vert_{1}  $. 
\item[(ii)]  $ \lambda TP_{(S(f)-\lambda e)^{+}}e\leq C \Vert f\Vert_{1}  $.
\end{enumerate}
\end{cor}
\begin{proof}
(i) Let $ f $ be an $ L^{1}(T)$-bounded martingale. Set $ L(f)=f^{*}= \displaystyle\sup_n\vert f_{n}\vert $. Then it suffices to show that it is to class $\mathcal{A}$. In fact, $ (1),(2) $  and $( i) $ of $ (3) $ are obviously satisfied and by  $[$ \cite{GRJ}, Theorem 6.5] 

$$  
T \left(\sup_{k\leq n}\vert f_k\vert\right)\leq C T\vert f_n\vert \leq C \sup_n T\vert f_n\vert.
$$
 Then $$ T \left(\sup_{ n}\vert f_k\vert\right)\leq C \sup_n T\vert f_n\vert.$$
 So  $ f^{*} $ is a class $ \mathcal{A} $. This proof is a completed by theorem \ref{theo 4.4}.
 
(ii) Similarly to (i). Put   $ L(f)=S(f) $, it is easy to see that $$ \vert L(f+g)\vert\leq C( \vert L(f)\vert + \vert L(g) \vert ).$$ 
 Then $ (1) $ and $ (ii) $ of $ (3) $ are obviously satisfied.
For  property $ (2)$, note that 
$P_{S_n} \uparrow P_{S(f)}.$
 So $$P_{f^*} \geq P_{\sqrt{n} S_n}=P_{S_n}.$$
  We deduce that $$P_{f^*} \geq P_{S(f)}.$$
   It remains  $ (i) $ of $ (3) $. We have
    $$ \displaystyle\sum_{k=1}^{\infty} (\Delta f_{k})^{2}\leq   \left( \sum_{k=1}^{\infty} \Delta f_{k} \right)   ^{2}. $$
    Then $$ \displaystyle \Vert S(f) \Vert _{1}\leq \Vert \sum_{k=1}^{\infty} \vert \Delta f_{k} \vert \Vert_{1}.$$ 
    Thus $ S(f) $ is a class $ \mathcal{A} $ and by Theorem \ref{theo 4.4} the proof is finished.
\end{proof}
From corollary, we try to major the first member of (i) by a quantity which depends on an operator of class $ \mathcal{A} $.

\begin{definition}\label{definition 4.6}
Consider a martingale denoted as $(f_n)$ on the space $E$. Let $(v_n)$ be a sequence such that each $v_n$ belongs to $R(T_{n-1})$. We define a martingale transform as the sum $$h_{n}=\displaystyle\sum_{k=1}^{n}v_{k}\Delta f_{k}$$ where it is required that $v_{n}\Delta f_{n}\in E^{u}$.
\end{definition}

\begin{lemma}\label{lemma 4.7}
Let $ h $ be a martingale transform of an $L^{1}(T)$-bounded martingale with $ v_k\leq Me$ assume that $ h^{*} , f^{*} \in E_{s}$. Then,
\begin{enumerate}
\item[(i)] $P_{h^{*}}e\leq P_{f^{*}}e $ and $ P_{S(h)}e\leq P_{f^{*}}e$. 
 \item[(ii)] $ h^{*} $
and $ S(h) $ are class $ \mathcal{A} $.
 \end{enumerate}
  \end{lemma}
\begin{proof}
(i) Let $ x\in B_{f^{*}}^{d}$, 
we have
 $$ x\wedge f^{*}=0 $$ 
 implies for each $ n $ $$ \vert x \vert \wedge \vert f_{n}-f_{n-1}\vert \leq \vert x\vert \wedge 2 f^* =0$$  
 then
$$ P_{M\vert f_{n}-f_{n-1}\vert}  x
=0 $$ 
thus 
$$  P_{\sum_{k=1}^{n} v_{k}\Delta f_{k}} x \leq \sum_{k=1}^{n}P_{v_k\Delta f_k} =0$$ 
finally   $ x\in B_{h^{*}}^{d}$ and $$P_{h^{*}}e\leq P_{f^{*}}e.$$
For the second, note that $ S(h)\leq M S(f)$. Then  $$P_{S(h)}e\leq C P_{f^{*}}e.$$ 
(ii) By $ (i) $ we have obvious $ h^{*} $
and $ S(h) $ are class $ \mathcal{A} $.
\end{proof}
\begin{cor} \label{Corollary 4.8}
 Let $f$ be an  $ L^1(T) $-bounded martingale. If $ h $ is a  martingale transform of $f$, then 
$$ \lambda TP_{(h^{*}-\lambda e)^{+}}e \leq C \Vert f\Vert_{1}, $$
and
$$ \lambda TP_{(S(h)-\lambda e)^{+}}e \leq C \Vert f\Vert_{1}. $$
\end{cor}

Before giving the last  applications in this paper   we define the Rademacher functions in Riesz space. We remind you that $ \Delta_{n}^k $ are the dyadic sub-intervals of the interval $ [0, 1] $ if for each $ n= 1, 2, . . . \; \text{and} \; k = 1, 2, . . . , 2^n $
 $$   \Delta_{n}^k= \left( \frac{k-1}{2^n} ,\frac{k}{2^n} \right) . $$
 The Rademacher functions  $ r_{n}(t) $ are defined  on closed interval  $ [0,1] $ by:
$$
r_{n}(t)=
\left\{\begin{array}{ll}
		(-1)^{k-1}\; & \; \text{if} \; t\in  \Delta_{n}^k\\
		0  \; & \; \text{if} \; t=\frac{k}{2^{n}}
	\end{array}\right.
$$
From this definitions, the possibles valued are $ 1,-1 $ and $ 0 $,furthermore  $ r_{n} $  is non null if $ t\neq k2^n  $ (see\cite{RF}). By this remark we extend this function in Riesz space as follows.  
\begin{definition}\label{definition 4.9}
The Rademacher function in Riesz space $ \mathfrak{R} _{n} $ on $ [0 , 1] $ are given by $ \mathfrak{R} _{n} (t)= r_{n}(t)e $ for each $ t\in [0 , 1] $, with $ r_{n}  $ design the classical Rademacher function.
\end{definition}
Therefore, for all $ a,b \in [0,1] $ we obtain 
$$  \int _{a}^b \mathfrak{R}_n (t)dt = \int _{a}^b r_{n}(t)e dt  =\left(  \int _{a}^b r_{n}(t)dt\right) e . $$ 
Then $ \mathfrak{R}_{n} \in \mathcal{RI} \left( [a, b], E_e\right)$  and by corollary 1.2 of \cite{RF} we deduce the following properties.
\begin{cor}\label{corollary 4.10}
Let $ \mathfrak{R}_n $ be the Rademacher function in Riesz spaces, then
\begin{enumerate}
\item[(i)] $ \mathfrak{R}_n $ is Riemann integrable on $ [0 ,1] $ and  $ \int_{0}^{1}\mathfrak{R}_{k}(t)dt=0 $.
\item[(ii)]  $ \int_{0}^{1}\mathfrak{R}_{n}(t)\mathfrak{R}_{m}(t)dt=e $ if $ n=m $ and $ \int_{0}^{1}\mathfrak{R}_{n}(t)\mathfrak{R}_{m}(t)dt=0 $ if $ n\neq m$.
\end{enumerate}
\end{cor}
\begin{proof}
The proof of this proposition is based on the of clasical case. For $ (i) $ we see
$$  \int _{0}^1 \mathfrak{R}_n (t)dt = \int _{a}^b r_{n}(t)e dt  =\left(  \int _{0}^1 r_{n}(t)dt\right) e =0 $$
\end{proof}

\begin{thm}\label{theo 4.11}
If $ f $ is $ L^{1}(T) $-bounded martingale, then
$$ \lambda TP_{(f^{*}-\lambda e)^{+}}e \leq C \Vert S(f)\Vert_{1}. $$
for any $\lambda >0$.
\end{thm}

\begin{proof} This theorem can be proved in two steps.

\textit{\textbf{Step 1}}. Let $ \mathfrak{R}_{k}(t) $ be a Rademacher  function in Riesz space. In this case $ \mathfrak{R}_{k}(t) $ take $ e $ and $ -e $ as possible valued, we have by using Hölder inequality, see \ref{theo 3.6}
$$
 \begin{aligned}
 \int_{0}^{1}T\vert \displaystyle \sum_{k}^{n} \mathfrak{R}_{k}(t)\Delta f_{k}\vert dt   & =  T\int_{0}^{1}\vert \displaystyle \sum_{k}^{n} \mathfrak{R}_{k}(t)\Delta f_{k}\vert\\
 &\leq T \left( \int_{0}^{1}\vert \displaystyle \sum_{k}^{n} \mathfrak{R}_{k}(t)\Delta f_{k}\vert ^{2}\right)^{\frac{1}{2}} \\
 &\leq TS_{n}(f)\\
\end{aligned}
$$

Then   $$  \int_{0}^{1} \sup T\vert \displaystyle \sum_{k}^{n} \mathfrak{R}_{k}(t)\Delta f_{k}\vert dt   \leq TS(f) $$

\textit{\textbf{Step 2}}. Let $ g=(g_{1},g_{2},\cdots) $ be a  martingale  defined as $ g_{n}=\displaystyle \sum_{k}^{n} \mathfrak{R}_{k}(t)\Delta f_{k} $.  
 Note that  $ f $ is a transform of $ g $, corollary \ref{Corollary 4.8} gives the following
     $$ \lambda TP_{(f^{*}-\lambda e)^{+}}e \leq 	C\Vert g \Vert_{1}  $$ 
   Integrating the two members
   
   $$ \begin{aligned}
   \lambda TP_{(f^{*}-\lambda e)^{+}}e &\leq 	C \int_{0}^{1}\sup T\vert g_{n}\vert dt\\
     &\leq C\Vert S(f) \Vert_{1}
   \end{aligned} $$
   This proof is  completed.
\end{proof}

\medskip

\end{document}